\documentclass[sigconf]{acmart}

\usepackage{booktabs} 
\usepackage{hyperref}
\usepackage{todonotes}
\usepackage{tcolorbox}

\usepackage{algorithm}
\usepackage[noend]{algpseudocode}

\def\VIZ#1{(\ref{#1})}      

\newtheorem{theorem}{Theorem}[section]
\newtheorem{proposition}{Proposition}

\newtheorem*{remark}{Remark}

\DeclareMathOperator{\supp}{supp}
\DeclareMathOperator{\svd}{svd}
\DeclareMathOperator{\diag}{diag}
\DeclareMathOperator{\sign}{sign}

\begin{document}
\title{Multiple Equivalent Solutions for the Lasso}

\author{Yannis Pantazis}
\affiliation{%
  \institution{IACM - FORTH}
  \streetaddress{Vassilika Vouton}
  \city{Heraklion}
  \country{Greece}
  \postcode{70013}
}
\email{pantazis@iacm.forth.gr}

\author{Vincenzo Lagani}
\authornote{Gnosis Data Analysis PC, Heraklion, Greece}
\affiliation{%
  \institution{CSD - UoC}
  \streetaddress{Voutes Campus}
  \city{Heraklion}
  \country{Greece}
  \postcode{70013}
}
\email{vlagani@csd.uoc.gr}

\author{Paulos Charonyktakis}
\affiliation{%
  \institution{Gnosis Data Analysis PC}
  \streetaddress{Vassilika Vouton}
  \city{Heraklion}
  \country{Greece}
  \postcode{70013}
}
\email{haronykt@gmail.com}

\author{Ioannis Tsamardinos}
\authornote{Gnosis Data Analysis PC, Heraklion, Greece, IACM/FORTH, Heraklion, Greece and Huddersfield University, Yorkshire, UK}
\affiliation{%
  \institution{CSD - UoC}
  \streetaddress{Voutes Campus}
  \city{Heraklion}
  \country{Greece}
  \postcode{70013}
}
\email{tsamard@csd.uoc.gr}

\renewcommand{\shortauthors}{Pantazis et al.}

\begin{abstract}
Feature selection is an important problem studied in data analytics seeking to identify a minimal-size feature subset that is optimally predictive for an outcome of interest. It is also a powerful tool in Knowledge Discovery as a means for  gaining domain insight, e.g., identifying which medical quantities carry unique information for the disease status. It is arguably less recognized however, that the problem may have multiple, equivalent solutions. In that case, it is misleading to domain experts to report only one of them and ignore all other equivalent solutions. In this paper, we extend a well-established single, feature selection algorithm (i.e., reporting a single solution), namely the Lasso algorithm, to the multiple solution problem based on formalized notion of equivalence for both classification and regression tasks. Empirical results are obtained using a fully automated pipeline called Just Add Data Bio  or JAD Bio training and selecting multiple, linear as well as nonlinear learners, optimizing hyper-parameter values, and correcting for the bias of multiple inductions (model selection). The results show that multiple solutions do exist in real datasets, as well as the ability of the algorithm to identify a subset of them. A comparison with the Statistical Equivalent Solutions (SES) algorithm shows that Lasso equivalent solutions have better prediction performance at the cost of selecting more features.
\end{abstract}


\maketitle

\section{Introduction}
The problem of supervised Feature Selection (FS) (a.k.a. variable selection) has been studied for several decades in data science fields, such as statistics, machine learning, and data mining. Informally, the problem can be defined as selecting a subset of the available feature set, such that it is of minimal-size and at the same time maximally-predictive for an outcome of interest. Feature selection is performed for several reasons. It may reduce the cost or risk of measuring, computing, storing, and processing of the features. It leads to models of smaller dimensionality that are easier to visualize, inspect, and understand. It may even result in more accurate models by removing the noise, treating the curse-of-dimensionality, and facilitating model fitting. Arguably however, feature selection is primarily employed as a means for knowledge discovery and gaining intuition into the mechanisms generating the data. Indeed, a deep theoretical connection with causality has been identified \cite{Tsamardinos2003}. Actually, it is often the case that FS is the primary goal of an analysis and the learned model is only a by-product. For example, a medical researcher analyzing their molecular data may not care about the classification model to cancerous or healthy tissue; they can diagnose the tissues themselves without the use of the computer on the microscope. The whole point of their analysis is to identify the features (e.g., gene expressions) that carry all the information to perform the diagnosis. 

When FS is employed for knowledge discovery and understanding considering multiple, equivalent solutions is paramount. It is misleading to present to the domain expert a single feature subset and proclaim all other features are either redundant or irrelevant, if there are multiple equivalent solutions. In addition, when features have an associated measurement cost, one seeks the least cost solution. We argue that in these cases, the analyst ought to identify all solutions and present the possibilities to the domain expert. Multiple, equivalent FS solutions often exist in practice \cite{lagani2016feature}, especially in fields with low sample size, high dimensionality and noisy features. Biology and medicine are prototypical fields where these adverse conditions are met. In a seminal paper \cite{ein2004outcome}, Ein-dor and co-authors demonstrate that multiple, equivalent prognostic signature for breast cancer can be found just by analyzing the same dataset with a different partition in training and test set, showing that several genes exist which are practically exchangeable in terms of predictive power. Statnikov and Aliferis \cite{Statnikov2010} further show that the presence of multiple optimal signatures is not a rare occurrence, and is actually common in biological datasets.

Despite the relevance of the problem, to date only few algorithms address it directly. Two constraint-based approaches (constraints stemming from the results of conditionally independence tests) are the Target Information Equivalence (TIE$^*$ \cite{statnikov2013algorithms}) and the Statistically Equivalent Signatures (SES \cite{lagani2016feature}) algorithms. Another algorithm which has been recently proposed by Cox and Battey \cite{Cox2017pnas} searches for multiple solutions across a large number of separate analyses. Unfortunately, both the Cox--Battey method and TIE$^*$ are quite computationally intensive.

In this work, we show that multiple equivalent solutions with controlled performance deviations can be also defined and identified for the popular least absolute shrinkage and selection operator (Lasso) method \cite{Tibshirani1996}. Previous studies based on bootstrapping \cite{Bach2008}, sub-sampling \cite{Meinshausen2010} or locally perturbing the support of the Lasso solution \cite{Hara2017} propose to run the Lasso multiple times and collectively establish the best (unique) solution. Their goal is to improve the ordinary Lasso solution in terms of either accuracy or stability. However, this is irrelevant to our goal which is to determine the set of equivalent Lasso solutions.

The case when {\em strongly-equivalent Lasso solutions} (SELSs) exist has been developed and  studied in \cite{Tibshirani2013}. The author provided there various characterizations of the space of all SELSs. Multiple SELSs however, rarely occur in real datasets. Henceforth, we propose a relaxed definition of equivalence requiring the equivalent solutions to have not exactly the same, but a similar (within a tolerance threshold) in-sample loss. This is useful because with finite sample different feature subsets may appear statistically indistinguishable in terms of loss, even though asymptotically there is a single optimal solution. Specifically, using the root mean squared error (RMSE) as the loss function, we define the {\em RMSE-equivalent Lasso solutions} (RELSs). In contrast, a natural loss for $l_1$-penalized logistic regression (i.e., logistic Lasso) is the deviance leading to {\em Deviance-equivalent Lasso solutions} (DELSs). 
Inspired by the characterization of the SELS space, we devise an algorithm that computes a maximal subspace of the space of relaxed Lasso solutions. A set of constraints that guarantees sparsity as well as the key characteristics of the Lasso solution(s) is imposed. The proposed algorithm handles both RMSE and deviance equivalence in a unified manner and actually it is easily extensible to any convex performance metric. Furthermore, we derive theoretical bounds both for the RMSE and the deviance that relate the allowed tolerance with the spectral properties of the predictors matrix.

We perform empirical experiments on several real datasets providing corroborating evidence of the prevalence of multiple, equivalent solutions in real scenarios and quantifying the efficacy of the proposed Lasso extension. Indeed, the presence of multiple heterogeneous signatures (i.e., solutions) having at the same time low value for their coefficient of variation indicates that multiple, equivalent solutions are common across several application fields.
In addition, we compare the proposed algorithm against the SES algorithm that can scale w.r.t. the dimensionality of the data. A prominent difference is that Lasso equivalent solutions have on average better prediction performance at the cost of selecting larger signatures while the SES algorithm is more parsimonious.
A novelty of the comparison is the employment of a fully automated analysis pipeline (called Just Add Data Bio or JAD Bio) that tries both linear and non-linear learners, optimizes their hyper-parameter values, and corrects the estimation of performance for the bias of the tuning. Thus, the results are obtained after matching each feature selection with the most appropriate learner and its hyper-parameter values.

\section{Definition of Lasso Solutions Equivalence}
\label{def:lasso:equiv:sec}

\subsection{Lasso preliminaries}
We first introduce the Lasso inference problem. Given an outcome vector (a.k.a. target variable) $y\in\mathbb R^n$, a matrix $X\in \mathbb R^{n\times p}$ with the predictor variables and the coefficient vector $\beta\in\mathbb R^p$, the Lagrangian form of Lasso is defined as
\begin{equation}
	\min_{\beta\in\mathbb{R}^p} L(\beta)
	:= \min_{\beta\in\mathbb{R}^p} \frac{1}{2}||y-X\beta||_2^2 + \lambda ||\beta||_1
	\label{lasso:eq}
\end{equation}
where $\lambda\geq 0$ is a penalty parameter. Lasso solvers such as LARS \cite{Efron2004} and FISTA \cite{Beck2009} have been extensively applied in thousands of real problems while extensions of \VIZ{lasso:eq} to generalized linear models (GLMs) such as the logistic regression with $l_1$-norm regularization have been also proposed \cite{Friedman2010}. In GLM Lasso, the quadratic cost term is replaced by the deviance which essentially behaves as the negative log-likelihood.

\vspace{-2mm}
\subsection{Strong equivalence}
\vspace{-1mm}
We briefly present the strong equivalence and its key properties which have been firstly defined and developed in \cite{Tibshirani2013}. It has been proved that when the entries of $X$ are drawn from a continuous probability distribution, the uniqueness of the Lasso solution is almost surely guaranteed \cite{Tibshirani2013}. However, when the predictor variables take discrete values or there exists perfect collinearity then several optimal Lasso solutions may exist. Mathematically, 
given $\lambda>0$, two vectors $\hat{\beta},\hat{\beta}'\in\mathbb R^p$ are Lasso equivalent solutions in the strong sense if and only if $L(\hat{\beta}) = L(\hat{\beta}')$. The set of all SELS is defined as
\begin{equation}
K := \{x\in\mathbb R^p : L(x) = \min_\beta L(\beta)\} \ ,
\end{equation}
Additionally, two SELSs not only share the same cost function but also predict the same values for the target variable \cite{Tibshirani2013} (i.e., $X\hat{\beta} = X\hat{\beta}'$). Consequently, SELSs have the same $l_1$ norm, $||\hat{\beta}||_1 = ||\hat{\beta}'||_1$. Another characteristic of all SELSs is that the non-zero coefficients are not allowed to flip their sign \cite{Tibshirani2013}. Hence, if $\hat{\beta}_{i}>0$ for the $i$-th coefficient of a Lasso solution then $\hat{\beta}_{i}'\ge0$ for any SELS, $\hat{\beta}'$, and, similarly, for the coefficients with negative values.

Proceeding, let $\hat{\beta}$ be the Lasso solution with the largest active set (or support), $\mathcal E:=\supp(\hat{\beta}):=\{i:\hat{\beta}_i\neq 0\}$ which is also known as the equicorrelation set \cite{Tibshirani2013}. The maximal Lasso solution is computed using either the elastic net with vanishing $l_2$ penalty norm or the variant of LARS described in \cite{Tibshirani2013}. 
With a slight abuse of notation due to the restriction of the solution space to the predictor variables with non-zero coefficients, the set of all SELSs can be rewritten as
\begin{equation}
K = \{x\in\mathbb R^{|\mathcal E|}: X_{\mathcal E} (x-\hat{\beta}_{\mathcal E})=0\ \& \ Sx\ge0\} \ ,
\label{polytope:def}
\end{equation}
where $X_{\mathcal E}$ is the matrix that contains only the columns of $X$ that are indexed by the equicorrelation set $\mathcal E$ while $S:=diag(s)$ is the diagonal matrix with the signs of the Lasso solution, $s:=\sign(\hat{\beta}_\mathcal E)$. The first constraint ensures that all elements in $K$ will have the same fitted value while the second constraint ensures that coefficients' sign will not flip.

\subsubsection{Enumeration Algorithm for SELSs}
A bounded polyhedron (i.e., a polytope) can be represented either as the convex hull of a finite set of vertices or by using a combination of linear constraint equalities and inequalities. In particular, the vertices of the convex polytope, $K$, defined above corresponds to the ``extreme'' SELSs. Thus, enumerating the vertices of $K$ from the set of equality and inequality constraints, we can enumerate all SELSs. There exist algorithms that enumerate the vertices defined by a set of inequality constraints (see Fukuda et al. \cite{Fukuda1997} and the references therein) and they can be extended to take into account equality constraints, too. In this paper, we employed the Matlab package by Jacobson \cite{Jacobson2015} which contains tools for converting between the (in)equality and the vertices representations.

\subsubsection{Cardinality of equivalent solutions}
It is noteworthy that the number of vertices, like the number of edges and faces, can grow exponentially fast with the dimension of the polytope making the enumeration algorithm impractical when the dimension is larger than 20. In such cases, we cannot enumerate all the extreme SELSs in reasonable time. Nevertheless, it would be useful to know which of the variables participate in all SELSs and which are not. 
In Section~\ref{disp:indisp:sec}, a practical categorization of the variables into dispensable (participate in some solutions) and indispensable (participate in all solutions) which has been firstly introduced in \cite{Tibshirani2013} is presented and extended to the relaxed equivalence, too.

\subsection{Relaxed equivalence}
A less restrictive equivalence is to require two solutions to have similar performance (loss, error). We will search for solutions whose performance metric differ by a small tolerance from a given solution. Moreover, in order to avoid handling absolute quantities, we suggest working with the relative performance. Denoting by $D(\cdot)$ the performance metric, $\hat{\beta}$ the given solution and $TOL$ the tolerance, we say that $\bar{\beta}$ is performance equivalent to $\hat{\beta}$ if and only if the following relation on the relative performance metric is satisfied
\begin{equation}
D(\bar{\beta}) \le (1+TOL) D(\hat{\beta}) \ .
\end{equation}

However, controlling only the performance metric may add unnecessary redundancy to the relaxed solutions destroying the desired property of sparsity. Indeed, a potentially large number of irrelevant or redundant variables may satisfy the performance constraint for a given $TOL$, nonetheless, they should not belong to the set of equivalent solutions. To alleviate this issue, we propose to restrict the space of relaxed solutions by allowing only the non-zero coefficients of the given solution to vary. Thus, for a given Lasso solution, $\hat{\beta}$ with support $\supp(\hat{\beta})=\{i: \hat{\beta}_i\neq 0\}$, a performance metric, $D(\cdot)$, and a tolerance $TOL$, we define the set of {\em D-equivalent Lasso solutions} as
\begin{equation}
K_D^{TOL}(\hat{\beta}) := \{\beta\in\mathbb R^p: \beta_{-\supp(\hat{\beta})} = 0 \ \  \&  \ \  D(\beta) \leq (1+TOL)D(\hat{\beta})\} \ ,
\label{relaxed:equiv:def}
\end{equation}
where $-A$ denotes the complement set of $A$.
The following proposition states that the convexity of the performance metric with respect to $\beta$ is inherited to the set of all relaxed Lasso solutions.

\begin{proposition}
Let $D(\cdot)$ be a convex performance metric. Then, the set $K_D^{TOL}(\hat{\beta})$ of all $D$-equivalent solutions is convex.
\end{proposition}

\begin{proof}
Let $\bar{\beta}_1,\bar{\beta}_2\in K_D^{TOL}(\hat{\beta})$ and $c_1,c_2\ge 0$ such that $c_1+c_2=1$, then for $\bar{\beta}:=c_1\bar{\beta}_1 + c_2\bar{\beta}_2$ we have $\supp(\bar{\beta}) \subset \supp(\hat{\beta})$ and
\begin{equation*}
\begin{aligned}
D(\bar{\beta}) &= D(c_1\bar{\beta}_1 + c_2\bar{\beta}_2)
\le c_1 D(\bar{\beta}_1) + c_2 D(\bar{\beta}_2) \\
&\le c_1 (1+TOL)D(\hat{\beta}) + c_2 (1+TOL)D(\hat{\beta}) \\
&= (1+TOL)D(\hat{\beta})
\end{aligned}
\end{equation*}
Thus, $\bar{\beta}\in K_D^{TOL}(\hat{\beta})$ and convexity is proved.
\end{proof}

Relaxed equivalence can be defined for several different performance metrics, which depend on the type of the analysis task (classification, regression, survival analysis, etc.). Different characteristics of the relaxed solution space can be conveyed by a suitable performance metric. For instance, when the Lasso cost function is used as a performance metric, the support of $\hat{\beta}$ is maximal and $TOL$ is set to 0 then the strong equivalence is obtained. Next, we present one performance metric for regression problems and one for classification tasks.

\subsubsection{RMSE equivalence for Lasso}
The root mean squared error defined by
$$RMSE(\beta) := \frac{1}{\sqrt{n}}||y-X\beta||_2$$
is a standard performance metric for regression problems. The set of RMSE-equivalent Lasso Solutions (RELSs), $K_{RMSE}^{TOL}(\hat{\beta})$, 
is a convex set since the RMSE is a convex performance metric. We note also that RMSE-equivalence with tolerance $TOL$ is almost equivalent with MSE-equivalence with tolerance $2TOL$ since $RMSE(\beta) \leq (1+TOL)RMSE(\hat{\beta})$ implies that $MSE(\beta) \leq (1+TOL)^2 MSE(\hat{\beta}) \approx (1+2TOL) MSE(\hat{\beta})$ for small values of the tolerance.

\subsubsection{Deviance equivalence for logistic Lasso}
Logistic Lasso is preferred as a feature selection method in binary classification tasks compared to the ordinary Lasso because it takes into consideration the distinguishing properties of the problem. In logistic Lasso, the quadratic term in \VIZ{lasso:eq} is replaced by the deviance. Hence, despite the fact that RMSE is applicable, a more appropriate performance metric is the mean deviance defined as
\begin{equation*}
DEV(\beta) := \frac{1}{n} \sum_{i=1}^n \left[-y_i x_i^T\beta + \log(1+\exp(x_i^T\beta))\right] \ .
\end{equation*}
The set of Deviance-equivalent Lasso Solutions (DELSs), $K_{DEV}^{TOL}(\hat{\beta})$, 
is also a convex set since the deviance is a convex function of its argument.

\section{Enumerating Relaxed Lasso Solutions}
\label{relaxed:mls:sec}
In the definition of relaxed equivalence (i.e., \VIZ{relaxed:equiv:def}), we limit the active set of relaxed equivalent solutions to a subset of the given solution's active set. Despite this restriction, the enumeration of all solutions is still difficult because of the shape of $K_D^{TOL}(\hat{\beta})$; the set continuously curves resulting in an infinite number of extreme solutions. Hence, we propose to restrict the solutions that we will eventually enumerate to a subset of $K_D^{TOL}(\hat{\beta})$. Inspired by the representations of SELSs, we propose an approach that enumerates the vertices of the largest convex polytope of the relaxed space that allows relative performance be less than $TOL$.
The central idea is to relax the Null space constraint while keeping intact the constraint on the sign of the coefficients. For this purpose, we rewrite the Null space constraint using the singular value decomposition for $X_{\mathcal E}$. Let $X_{\mathcal E}$ be decomposed as
\begin{equation}
X_{\mathcal E} = U_{\mathcal E} \Sigma_{\mathcal E} V_{\mathcal E}^T \ ,
\end{equation}
where $U_{\mathcal E}$ and $V_{\mathcal E}$ are orthogonal matrices with the left and right space eigenvectors, respectively, while $\Sigma_{\mathcal E}$ is a diagonal matrix with diagonal elements the ordered singular values (i.e., $\Sigma_{\mathcal E}:=\diag(\sigma_1,...,\sigma_{|\mathcal E|})$ with $\sigma_1\geq ... \geq\sigma_{|\mathcal E|}\geq 0$). Assume that  
$\sigma_{i}\neq0$ for $i=1,...,i^+$ and $\sigma_{i}=0$ for $i=i^++1,...,|\mathcal E|$, then, the constraint $X_{\mathcal E} (x-\hat{\beta}_{\mathcal E})=0$ in \VIZ{polytope:def} is equivalent to $V_{\mathcal E}^+ (x-\hat{\beta}_{\mathcal E})=0$ where $V_{\mathcal E}^+ = [v_1|...|v_{i^+}]$ corresponds to the matrix with the eigenvectors whose singular values are non-zero.

Inspired by the above representation, we propose to relax the constraint to
\begin{equation*}
V_{\mathcal E}^* (x - \hat{\beta}_{\mathcal E}) = 0
\end{equation*}
where $V_{\mathcal E}^*= [v_1|...|v_{i^*}]$ as above while $i^*$ is an integer between $1$ and $|\mathcal{E}|$ to be specified later. Thus, the convex polytope for the relaxed Lasso solutions is defined as
\begin{equation}
K^* := \{x\in\hat{\beta}_{\mathcal E}+[-l,l]^{|\mathcal E|}: V_{\mathcal E}^*(x - \hat{\beta}_{\mathcal E}) = 0 \ \& \ Sx\ge0\} \ .
\label{relax:polytope:def}
\end{equation}
We further constrain the relaxed Lasso solutions to live in a box and have coefficients that are at most $l$ far away from the given Lasso solution, $\hat{\beta}$. This constraint is added so as to guarantee the boundedness of the polyhedron defined by the other two constraint. We choose to set the box size to be $l=||\hat{\beta}||_\infty = \max_i \hat{\beta}_i$ which allows any coefficient to vary till the zero value. 
Since the relaxed polytope is performance metric ignorant, our goal is to choose appropriately $i^*$ so as $K^*$ is a maximal subset of the relaxed solution space, $K_D^{TOL}(\hat{\beta})$.
Figure~\ref{relax:lasso:sol:fig} demonstrates a two-dimensional example where the Lasso regression problem has a unique solution (small red circle). By setting $i^*=1$ the relaxed set of Lasso solutions constitutes the red line that connects the positive axes determined by the direction of $v_2$. The two vertices (solid red dots) are candidates as ``extreme'' relaxed Lasso solutions which are further tested for ensuring that their relative performance is below the maximum tolerance.
Before proceeding with an algorithm that enumerates the vertices of $K^*$ that is guaranteed to be a subset of the relaxed space, we derive theoretical bounds for two performance metrics.

\begin{figure}[t]
\begin{center}
\includegraphics[width=0.4\textwidth]{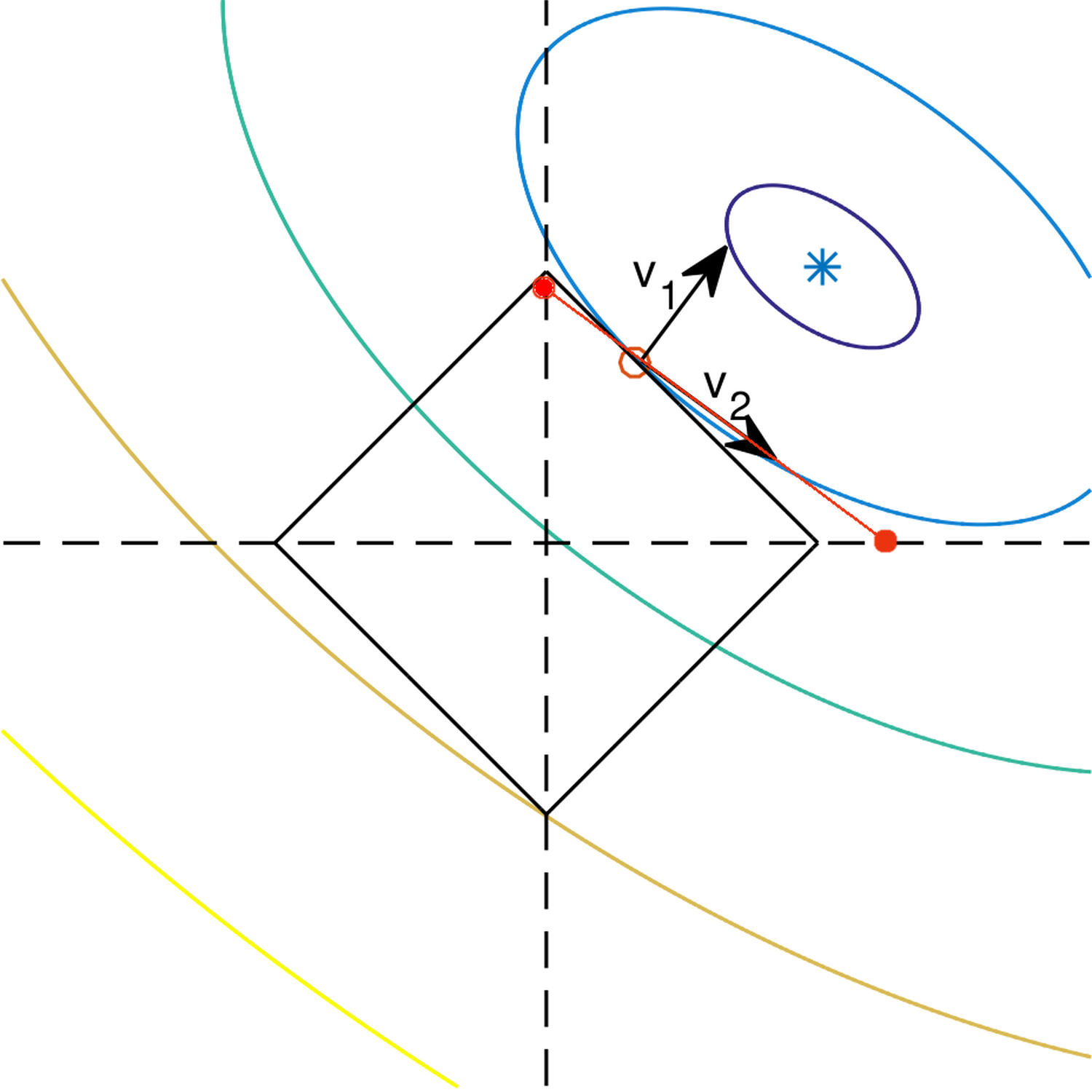}
\caption{
Two dimensional example of Lasso regression problem. The ellipses represent the contours of the objective function, while the rhomboid corresponds to the Lasso penalty. Moving along the direction of $v_2$, the component with the smallest singular value, the loss function is affected the least. By setting $i^* = 1$ we allow the algorithm to move along the direction of $v_2$ (red line) to search for candidate RELSs. 
}
\label{relax:lasso:sol:fig}
\end{center}
\end{figure}

\subsection{Theoretical Bounds}
\vspace{-1mm}
Theorems \ref{RMSE:DEV:thm} present error bounds for the RMSE and for the deviance. The bounds depend on the norm of the vector with the discarded singular values, $\bar{\sigma}^*:=[\sigma_{i^*+1},...,\sigma_{|\mathcal E|}]^T$. The proof of both parts can be found in the Appendix.

\begin{theorem} \label{RMSE:DEV:thm}
(a) Let $\hat{\beta}$ be a Lasso solution\footnote{Variants of the standard Lasso are also valid.}. For any $\bar{\beta}\in K^*$ it holds that
\begin{equation} \label{RMSE:bound}
RMSE(\bar{\beta}) \leq RMSE(\hat{\beta}) +  2 l  ||\bar{\sigma}^*||_\infty \sqrt{\frac{|\mathcal E|}{n}} \ .
\end{equation}

\noindent (b) Let $\hat{\beta}$ be a logistic Lasso solution. For any $\bar{\beta}\in K^*$ it holds that
\begin{equation} \label{DEV:bound}
DEV(\bar{\beta}) \leq DEV(\hat{\beta}) +  2 l ||\bar{\sigma}^*||_2 \sqrt{\frac{|\mathcal E|}{n}}  \ .
\end{equation}
\end{theorem}

\subsection{Enumeration algorithm}
Algorithm \ref{rels:enu:alg} reliably enumerates relaxed solutions that belong to a subset of the set of all relaxed Lasso solutions. It takes as input a Lasso solution, a convex performance metric $D(\cdot)$, a tolerance $TOL$ and the maximum dimension $ d_{max}$ of the subspace that will be explored. We assume that the provided Lasso solution has maximal support hence the equicorrelation set and the sign vector are directly estimated as in lines 2 and 3 of Algorithm \ref{rels:enu:alg}. In the {\bf for} loop, the dimension of the subspace is increased starting from $1$ until the maximum value is reached or there are vertices whose performance metric is above the tolerance. In other words, the algorithm removes the restriction stemming from the component $v$ with the least singular value, allowing itself to search along the $v$ direction. It sequentially removes restrictions corresponding to components with larger singular values. As the dimension of the exploration subspace is increased the possibility of finding solutions whose performance exceeds the tolerance is also increased. If such solutions occur then the algorithm breaks and the solutions that exceed the tolerance are discarded.  Due to the convexity of the performance metric, the polytope that is defined by the remaining solutions is also a subset of the relaxed space. Thus, the algorithm is sound.
 Additionally, a variant of binary search can be utilized instead of the linear search for identifying the optimal $i^*$. However, the performance gains in computational time is minimal due to the fact that the number of vertices grows exponentially fast since the computational cost for the last iteration is proportional to the total cost of all previous iterations.

\begin{algorithm}
	\caption{Relaxed Lasso Solution Enumeration}
	\label{rels:enu:alg}
	\begin{algorithmic}[1]
		\Procedure{RELSEnumeration}{$\hat{\beta}, D(\cdot), TOL, d_{max}$} 
		\State $\mathcal E=\supp(\hat{\beta})$
        \State $s = \sign(\hat{\beta}_{\mathcal E})$
		\State $[\cdot,\cdot,V_{\mathcal E}] = \svd(X_{\mathcal E})$
		\For {$i = 1 : d_{max}$}
        \State $i^* = |\mathcal E|-i$
		\State $V^* = [v_1|...|v_{i^*}]$
		\State $\bar{\beta}^{all} = \text{compute\_vertices} (V^*, s)$.
		\If{$\exists\bar{\beta}\in\bar{\beta}^{all}$ with $D(\bar{\beta})> (1+TOL)D(\hat{\beta})$}
		\State \textbf{break}
		\EndIf
		\EndFor
		\State \textbf{return} only the D-equivalent $\bar{\beta}$'s
		\EndProcedure
	\end{algorithmic}
\end{algorithm}

\begin{remark}
\normalfont We also investigated alternative definitions of the subset, $K^*$, where we relax different properties of the strongly-equivalent characterization. For instance, we relaxed the condition $X(\hat\beta-\bar\beta)=0$ to $||X(\hat\beta-\bar\beta)||<=r$ with $r>0$ resulting, however, in optimization problems which are algorithmically intractable. Overall, the choice of performance measures (RMSE and Deviance) and the particular subset, $K^*$, in (\ref{relax:polytope:def}), were based on the mathematical simplicity and the practical implementation both being very critical for the adoption of a less popular but important idea of searching for and enumerating multiple equivalent solutions.
\end{remark}

\subsection{Determining the reference Lasso solution}
Algorithm \ref{rels:enu:alg} searches for equivalent solutions in the least sensitive directions of the performance metric. It may discard features from the reference support but it cannot include features not initially selected. Therefore, it is important to provide an initial solution with the maximum support. We compute the reference Lasso solution $\hat{\beta}$, by optimally tuning its penalty parameter. We perform cross-validated hyper-parameter tuning using AUC (resp. MAE) as performance metric for classification (resp. regression) problems. It has been highlighted and proved \cite[Proposition 1]{Meinshausen2006} that many noise features are included in the prediction-oracle solution (i.e., the reference solution we estimate). Indeed, the value of the penalty parameter is lower in the prediction-oracle solution than the consistent one hinting towards larger support. We also observe similar behavior in our experiments with real datasets in the Results section. Thus, we expect the respective initial support to contain most, if not all, of the related to the target variable features. Of course, alternative methods based on elastic net or bootstrapping for selecting the reference solution/support can be utilized, however, such an exploration is beyond this paper's scope and it is left as future work.

\section{Variable Categorization}
\label{disp:indisp:sec}
Due to the exponential growth of the polytope's vertices, the complete enumeration is not always feasible. Nevertheless, there is an alternative approach to qualitatively assess the predictor variables by summarizing the interesting findings, e.g., reporting the set of features that belong to all solutions or only to some.
Exploiting the fact that the feasible range of a coefficient's value for all Lasso solutions can be efficiently computed as it was shown in \cite{Tibshirani2013}, a variable categorization is possible. 
We present the existing formulation for SELSs and based on it we extend it for RELSs.

\noindent {\bf SELSs.} For each $i\in\mathcal E$,  the $i$-th coefficient's lower bound $\hat{\beta}_i^{l}$ and upper bound $\hat{\beta}_i^{u}$ are computed by solving the linear programs
$$ \hat{\beta}_i^{l} = \min_x x_i \ \text{subject to}\ X_{\mathcal E}x=X_{\mathcal E} \hat{\beta}_{\mathcal E}\ \&\ Sx\ge0 \ ,$$
and
$$ \hat{\beta}_i^{u} = \max_x x_i \ \text{subject to}\ X_{\mathcal E}x=X_{\mathcal E} \hat{\beta}_{\mathcal E}\ \&\ Sx\ge0 \ ,$$
respectively. If $0$ is an element of the set $[\hat{\beta}_i^{l},\hat{\beta}_i^{u}]$ then the $i$-th variable is called dispensable otherwise it is called indispensable and they participate in all Lasso solutions. The fact that the sign of a coefficient remains the same in all SELSs implies that dispensable variables have either $\hat{\beta}^{l}=0$ or $\hat{\beta}^{u}=0$ while indispensable variables have either $\hat{\beta}^{l}>0$ or $\hat{\beta}^{u}<0$. From a practical perspective, the number of linear programs to be solved is $2|\mathcal{E}|$ which is feasible.


\noindent {\bf Relaxed equivalence.} Similarly, we can relax the criterion on dispensable/indispensable variables. Since the computational cost of the linear programs that are solved is manageable, we can discard the constraint on the control of the subspace dimension (i.e., there is no need for $d_{max}$). The linear programs for the lower and upper bounds for a given $i^*$ are
$$ \bar{\beta}_i^{l} = \min_x x_i \ \text{s.t.} \ V_{\mathcal E}^*x=V_{\mathcal E}^* \hat{\beta}_{\mathcal E}, \ -l\le x-\hat{\beta}_{\mathcal E} \le l \ \&\ Sx\ge0 \ , $$
and
$$ \bar{\beta}_i^{u} = \max_x x_i \ \text{s.t.} \ V_{\mathcal E}^*x=V_{\mathcal E}^* \hat{\beta}_{\mathcal E}, \ -l\le x-\hat{\beta}_{\mathcal E} \le l \ \&\ Sx\ge0 \ , $$
respectively. Evidently, as we increase $i^*$ the number of dispensable variables is increased while the number of indispensable variables is decreased. 

While in this paper we follow the terminology of "dispensable'' and ``indispensable'' variables from \cite{Tibshirani2013}, we would like to emphasize that {\em it may be misleading}. It implies that dispensable variables could be dispensed with and filtered out in some sense. Consider however, two features $W$ and $Z$ that are perfect copies of each other and at the same time, the strongest predictors. According to the definition they are dispensable but neither can be filtered out without a measurable loss in performance. Instead, we have recently proposed the terms ``indispensable'' and ``replaceable'' \footnote{Invited talk at KDD 2017 conference ``Advances in Causal-Based Feature Selection'' (http://mensxmachina.org/en/presentations/).} to clearly indicate that they the latter features can be replaced, not filtered out.

\section{Experimental Setup}
\label{exp:setup}


To evaluate the proposed method and compare against the SES algorithm, we used 9 datasets (Table~\ref{table:datasets}) trying to include a large variance of sample and feature sizes. No dataset was excluded from the evaluation based on the results. Seven datasets have binary outcome (i.e., classification tasks), while the other two have continuous outcome (i.e., regression tasks). The analyses were performed using 
the Just Add Data Bio v0.7 (JAD Bio; Gnosis Data Analysis; www.gnosisda.gr) 
automated predictive analysis engine.

\begin{table}[t]
\small
\centering
\caption{Dataset Overview.}
\label{table:datasets}
\vspace{-3mm}
\begin{tabular}{cccc}
\textbf{Dataset} & \textbf{Sample size} & \textbf{Feature Size}  & \textbf{Field} \\ \hline
\multicolumn{4}{c}{Classification} \\ \hline
arcene & 200 & 10000  & Bioinformatics \\
bankruptcy & 7063 & 147  & Finance \\
Ovarian & 216 & 2190 & Bioinformatics \\
Parkinson & 195 & 22  & Bioinformatics\\
prostate & 102 & 5966  & Bioinformatics\\
secom & 1567 & 590  & Industry \\
vV & 224 & 70  & Bioinformatics \\
\hline \multicolumn{4}{c}{Regression} \\ \hline
BC\_Continuous & 286 & 22282  & Bioinformatics \\
MP & 4401 & 202  & Materials \\
\end{tabular}
\end{table} 

\subsection{Just Add Data Bio Engine}
JAD Bio employs a fully-automated machine learning pipeline for producing a predictive model given a training dataset, and an estimate of its predictive performance. JAD Bio performs multiple feature selection and predictive modeling. It tries several configurations, i.e., combinations of preprocessing algorithms, feature selection algorithm, predictive modeling algorithms, and values of their hyper-parameters using a grid-search in the space of hyper-parameters. The best configuration is determined using K-fold Cross Validation, which is then used to produce the final models for each feature subset on all data. The cross-validated performance of the best configuration is known to be optimistic due to the multiple tries \cite{tsamardinos2015performance}; JAD Bio  estimates and removes the optimism using a bootstrap method before returning the final performance estimate \cite{tsamardinos2017bootstrapping}. The final performance estimates of JAD Bio are actually slightly conservative.

Specifically in terms of algorithms, JAD Bio trains several basic and advanced, linear and non-linear, multivariate machine learning and statistical models; namely, for classification problems trains Support Vector Machine models (SVMs) with linear, full polynomial, and Gaussian kernels, Ridge Logistic Regression models, Random Forests models, and Decision Trees, while for regression problems trains Ridge Linear Regression models instead of the logistic ones. In the following set of experiments more than 500 configurations were tried each time to find the optimal ones, resulting in more than 10000 models trained to compare two feature selection algorithms over nine datasets. 
JAD Bio is able to incorporate other user-defined preprocessing, transformation, and feature selection algorithms written in Java, R, or Matlab programming languages. This functionality allowed us to incorporate the proposed feature selection method into the pipeline.

\subsection{Evaluation Pipeline and Tuning}
Two variants of JAD Bio were used and compared: the original with the SES algorithm for the feature selection and an alternative with a Matlab implementation of the proposed method.
To evaluate and compare them we split each dataset in two stratified parts. Then, JAD Bio is applied on each half (as training set) for obtaining a set of predictive models and their performance estimation, corresponding to different feature subsets (i.e., solutions found by the feature selection algorithms). The trained models are then tested on the remaining data (as test set) to compute the performance on new data. We then test whether the performance of the solutions found are actually similar. Since we partition the data into two, equivalence of the performance of the feature selection subset is determined using the same sample size as in training and thus having the same statistical power. The new proposed method was tuned with penalty parameter, $\lambda$, taking values from $10^{-3}$  to $10^{3}$ with multiplicative step $10^{0.1}$. We set the $TOL$ value to 0.01 allowing Lasso solutions having up to 1\% performance error in the training set.

\section{Results}
\label{results:sec}

\begin{table*}[t]
\small
\centering
\caption{Lasso and SES comparison. For each dataset, split and method (Lasso/SES), the table reports the training set sample size, the number of identified signatures, the coefficient of variation (CoV) across the multiple signatures for the hold-out performance (AUC for classification, R$^2$ for regression), the CoV for the number of selected variables, the average hold-out performance (along with the two-tail, t-test adjusted p-value assessing the difference between SES and Lasso), as well as the average number of selected variables (along with the two-tail, t-test adjusted p-value assessing that the average number of variable selected by Lasso is different from the fixed number of variables selected by SES). Classification datasets are reported first, while the regression ones are at the bottom.}
\label{table:classificationResults}
\vspace{-3mm}
\begin{tabular}{ccccccccccc}  
\textbf{dataset} & \textbf{split} & \textbf{method} & \textbf{\begin{tabular}[c]{@{}c@{}}sample \\ size\end{tabular}} & \textbf{\# signatures} & \textbf{\begin{tabular}[c]{@{}c@{}}CoV \\ perform.\end{tabular}} & \textbf{\begin{tabular}[c]{@{}c@{}}CoV \\ vars\end{tabular}} & \textbf{\begin{tabular}[c]{@{}c@{}}average \\ perform.\end{tabular}} & \textbf{\begin{tabular}[c]{@{}c@{}}p-value \\ perform.\end{tabular}} & \textbf{\begin{tabular}[c]{@{}c@{}}average \\ \# vars\end{tabular}} & \textbf{\begin{tabular}[c]{@{}c@{}}p-value \\ \# vars\end{tabular}} \\\hline
arcene           & split1         & Lasso           & 100                                                             & 81                     & 0.0099                                                             & 0.0198                                                      & 0.8524                                                                  &                                                                         & 56.5                                                                &                                                                     \\
arcene           & split1         & SES             & 100                                                             & 100                    & 0.0373                                                             &                                                             & 0.7394                                                                  & $\leq$ 0.0001                                                           & 5                                                                   & $\leq$ 0.0001                                                       \\
arcene           & split2         & Lasso           & 100                                                             & 27                     & 0.0128                                                             & 0.0326                                                      & 0.8011                                                                  &                                                                         & 29.1                                                                &                                                                     \\
arcene           & split2         & SES             & 100                                                             & 100                    & 0.0088                                                             &                                                             & 0.7498                                                                  & $\leq$ 0.0001                                                           & 27                                                                  & $\leq$ 0.0001                                                       \\
bankruptcy       & split1         & Lasso           & 3531                                                            & 32                     & 0.0005                                                             & 0.022                                                       & 0.9704                                                                  &                                                                         & 51.8                                                                &                                                                     \\
bankruptcy       & split1         & SES             & 3531                                                            & 6                      & 0.0004                                                             &                                                             & 0.9615                                                                  & $\leq$ 0.0001                                                           & 17                                                                  & $\leq$ 0.0001                                                       \\
bankruptcy       & split2         & Lasso           & 3532                                                            & 40                     & 0.0005                                                             & 0.0126                                                      & 0.9697                                                                  &                                                                         & 72.1                                                                &                                                                     \\
bankruptcy       & split2         & SES             & 3532                                                            & 2                      & 0.0006                                                             &                                                             & 0.9615                                                                  & 0.0435                                                                  & 14                                                                  & $\leq$ 0.0001                                                       \\
Ovarian          & split1         & Lasso           & 107                                                             & 17                     & 0.0021                                                             & 0.0309                                                      & 0.9803                                                                  &                                                                         & 35.7                                                                &                                                                     \\
Ovarian          & split1         & SES             & 107                                                             & 6                      & 0.0065                                                             &                                                             & 0.9555                                                                  & 0.0003                                                                  & 6                                                                   & $\leq$ 0.0001                                                       \\
Ovarian          & split2         & Lasso           & 109                                                             & 3                      & 0.0047                                                             & 0.0777                                                      & 0.9654                                                                  &                                                                         & 19.7                                                                &                                                                     \\
Ovarian          & split2         & SES             & 109                                                             & 2                      & 0.0008                                                             &                                                             & 0.9647                                                                  & 0.8298                                                                  & 7                                                                   & 0.00641                                                             \\
Parkinson        & split1         & Lasso           & 97                                                              & 2                      & 0.0124                                                             & 0.0566                                                      & 0.9327                                                                  &                                                                         & 12.5                                                                &                                                                     \\
Parkinson        & split1         & SES             & 97                                                              & 2                      & 0.0026                                                             &                                                             & 0.9077                                                                  & 0.2421                                                                  & 1                                                                   & 0.03319                                                             \\
Parkinson        & split2         & Lasso           & 98                                                              & 1                      &                                                                    &                                                             & 0.9018                                                                  &                                                                         & 8                                                                   &                                                                     \\
Parkinson        & split2         & SES             & 98                                                              & 2                      & 0.0014                                                             &                                                             & 0.857                                                                   &                                                                         & 2                                                                   &                                                                     \\
prostate         & split1         & Lasso           & 51                                                              & 6                      & 0.0213                                                             & 0.0456                                                      & 0.9282                                                                  &                                                                         & 22.7                                                                &                                                                     \\
prostate         & split1         & SES             & 51                                                              & 3                      & 0.0276                                                             &                                                             & 0.9677                                                                  & 0.1554                                                                  & 3                                                                   & $\leq$ 0.0001                                                       \\
prostate         & split2         & Lasso           & 51                                                              & 1                      &                                                                    &                                                             & 0.9215                                                                  &                                                                         & 7                                                                   &                                                                     \\
prostate         & split2         & SES             & 51                                                              & 7                      & 0.0226                                                             &                                                             & 0.9268                                                                  &                                                                         & 2                                                                   &                                                                     \\
secom            & split1         & Lasso           & 783                                                             & 80                     & 0.0124                                                             & 0.0257                                                      & 0.6934                                                                  &                                                                         & 45.3                                                                &                                                                     \\
secom            & split1         & SES             & 783                                                             & 100                    & 0.0151                                                             &                                                             & 0.6071                                                                  & $\leq$ 0.0001                                                           & 15                                                                  & $\leq$ 0.0001                                                       \\
secom            & split2         & Lasso           & 784                                                             & 38                     & 0.0126                                                             & 0.0088                                                      & 0.6914                                                                  &                                                                         & 124.7                                                               &                                                                     \\
secom            & split2         & SES             & 784                                                             & 100                    & 0.0146                                                             &                                                             & 0.5741                                                                  & $\leq$ 0.0001                                                           & 12                                                                  & $\leq$ 0.0001                                                       \\
vV               & split1         & Lasso           & 112                                                             & 2                      & 0.0255                                                             & 0.1286                                                      & 0.6794                                                                  &                                                                         & 11                                                                  &                                                                     \\
vV               & split1         & SES             & 112                                                             & 4                      & 0.0309                                                             &                                                             & 0.6518                                                                  & 0.2422                                                                  & 4                                                                   & 0.0985                                                              \\
vV               & split2         & Lasso           & 112                                                             & 3                      & 0.0421                                                             & 0.1083                                                      & 0.6792                                                                  &                                                                         & 5.3                                                                 &                                                                     \\
vV               & split2         & SES             & 112                                                             & 2                      & 0.0017                                                             &                                                             & 0.6619                                                                  & 0.441                                                                   & 5                                                                   & 0.4226                                                              \\\hline
BC\_Continuous   & split1         & Lasso           & 143                                                             & 12                     & 0.0095                                                             & 0.0584                                                      & 0.8771                                                                  &                                                                         & 13.6                                                         &                                                                     \\
BC\_Continuous   & split1         & SES             & 143                                                             & 6                      & 0.0017                                                             &                                                             & 0.8851                                                                  & 0.0091                                                                  & 12                                                                  & $\leq$ 0.0001
                                                            \\
BC\_Continuous   & split2         & Lasso           & 143                                                             & 3                      & 0.0013                                                             & 0.0769                                                      & 0.8916                                                                  &                                                                         & 13                                                                  &                                                                     \\
BC\_Continuous   & split2         & SES             & 143                                                             & 2                      & 0.007                                                              &                                                             & 0.8743                                                                  & 0.1501                                                                  & 14                                                                  & 0.2254                                                         \\
MP               & split1         & Lasso           & 2200                                                            & 75                     & 0.0017                                                             & 0.0095                                                      & 0.5435                                                                  &                                                                         & 106.8                                                              &                                                                     \\
MP               & split1         & SES             & 2200                                                            & 24                     & 0.0034                                                             &                                                             & 0.4763                                                                  & $\leq$ 0.0001                                                           & 17                                                                  & $\leq$ 0.0001
                                                           \\
MP               & split2         & Lasso           & 2201                                                            & 70                     & 0.0033                                                             & 0.0184                                                      & 0.5469                                                                  &                                                                         & 53.7                                                         &                                                                     \\
MP               & split2         & SES             & 2201                                                            & 3                      & 0.0011                                                             &                                                             & 0.5117                                                                  & $\leq$ 0.0001                                                           & 21                                                                  & $\leq$ 0.0001
                                                          
\end{tabular}
\end{table*}

\subsection{Multiple, equivalent solutions are common for predictive analytics tasks}
Table \ref{table:classificationResults} reports the Lasso and SES results. Both algorithms identify multiple solutions on almost all datasets and splits, indicating that \emph{the presence of multiple solutions is a common problem across several application domains}. 

At the same time, the coefficient of variation (CoV, defined as the ratio between standard deviation and average) for the hold-out performances is always quite low; the median CoV for Lasso is $0.0097$ and for SES it is $0.0049$ (maximum values are respectively $0.0421$ and $0.0373$). This means that the different solutions produce models with performance having $\sim 0.01$ standard deviation around the mean AUC value. \emph{These results support the claim that multiple solutions found by Lasso and SES are indeed equivalent in terms of predictive power}.

Interestingly, when SES retrieves a large number of signatures (i.e., solutions) so does Lasso, and vice versa. Both algorithms provide numerous solutions for the arcene and secom datasets, only a few solutions for Ovarian, Parkinson, prostate, vV, BC\_Continuous, and MP (an exception is the bankruptcy dataset). This is despite the fact that the two algorithms follow quite different approaches. Thus, the results provide evidence that the order of magnitude for the number of equivalent solutions is a characteristic of the data at hand.

\subsection{Lasso and SES trade-offs}
Figure \ref{figure:classificationResults} graphically represents the hold-out performances (x-axis, AUC for classification datasets and R$^2$ for regression ones) and number of selected features (y-axis) for each signature retrieved by Lasso (circles) and SES (triangles) for the first split of each dataset (distinguished by color). The figure shows that Lasso consistently selects up to one order of magnitude more variables than SES (Table \ref{table:classificationResults}). In nine out of $18$ comparisons ($9$ datasets $\times$ 2 splits each) Lasso also achieves better performances than SES (adjusted p-value $\leq 0.05$), with an average AUC difference of $0.034$ between the two methods across classification datasets and $0.028$ R$^2$ between regression datasets. In summary, the two algorithms are quite complementary, with SES providing more parsimonious models with a moderate cost in terms of predictive power achieving different trade-offs. SES sacrifices optimality (does not attempt to identify all predictive features, or as it is known in the respective literature the full Markov Blanket) to achieve scalability; however, it is easily applicable to different data types by equipping it with an appropriate conditional independence test (see a comparison of Lasso and SES-based feature selection for time-course data \cite{Tsagris2018}).

\subsection{Multiple solutions are heterogeneous in terms of included features}
Table \ref{table:signatureHeterogenuity} presents, for both Lasso and SES, the number of identified signatures categorized according to their size. For sake of clarity only results from the first split are presented, with the results on the second split following almost identical trends. By construction SES retrieves signatures with equal size, while Lasso identifies signatures that can have different sizes. For each group of equal-size signatures, we computed the average Jaccard index within the group as well as  the average number of solution-specific features across all pairs of signatures. The Jaccard index between two sets is defined as the ratio between the size of their intersection and the size of their union, and ranges from $0$ (disjoint sets) to $1$ (perfectly overlapping sets), with two equal-size sets sharing half of their elements achieving $0.33$. For computing the number of solution-specific features between two sets we first take their union and then subtract their intersection. The resulting features are either in one solution or in the other, but not in both. For the cases where numerous solutions are selected by SES (arcene and secom datasets), both the Jaccard index and the number of dispensable elements indicate a quite elevated heterogeneity, meaning that the multiple solutions proposed by SES include solutions quite different from each other. When only a handful of signatures are retrieved by SES, they tend to be more homogeneous. This is a consequence of the mechanisms used by SES for producing multiple signatures,
which constraints each signature to differ from its closest sibling only for one variable.
In contrast, Lasso solutions can have more than two signature-specific features even  when only two signatures are retrieved. This can be seen in several rows of Table \ref{table:signatureHeterogenuity}, e.g., for the arcene dataset (signature size $53$, $54$) and for the secom dataset (signature size $43$). Interestingly, in several cases Lasso solutions differ from each other on average by $5$ or more solution-specific features, confirming that the multiple solutions retrieved by Lasso differ from each other due to the joint replacement of several variables.

\section{Conclusion}

In this paper we present an algorithm for deriving multiple, equivalent Lasso solutions with theoretical guarantees. The algorithm is empirically evaluated using an automatic platform, JAD Bio, over several learners and tuning hyper-parameters; the same platform is employed for a comparative evaluation against the  SES algorithm. The results confirm that (a) multiple Lasso solutions are often found in real-world datasets, (b) these multiple solutions can substantially differ from each other, despite their similar performances, and (c) Lasso and SES provide comparable results, with the first being on average slightly more predictive at the cost of selecting a larger number of features.

\begin{table}[t]
\small
\centering
\caption{Feature heterogeneity in multiple signatures. For the first split of each dataset the table reports the number of signatures grouped by size. The average Jaccard index and the average number of solution-specific features between each pair of subsets are presented. Results from the second split follow similar patterns (not shown).}
\label{table:signatureHeterogenuity}
\vspace{-3mm}
\begin{tabular}{ccccc}
\textbf{dataset} & \textbf{\begin{tabular}[c]{@{}c@{}}signat. \\ size\end{tabular}} & \textbf{\# signat.} & \textbf{\begin{tabular}[c]{@{}c@{}}Jaccard \\ index\end{tabular}} & \textbf{\begin{tabular}[c]{@{}c@{}}sol. spec. \\ features\end{tabular}} \\\hline
\multicolumn{5}{c}{Lasso} \\ \hline
arcene           & 53                                                               & 2                   & 0.86                                                              & 8                                                                        \\
           & 54                                                               & 2                   & 0.83                                                              & 10                                                                       \\
           & 55                                                               & 6                   & 0.89                                                              & 6.4                                                                      \\
           & 56                                                               & 26                  & 0.92                                                              & 4.9                                                                      \\
           & 57                                                               & 31                  & 0.94                                                              & 3.5                                                                      \\
           & 58                                                               & 13                  & 0.97                                                              & 2                                                                        \\
           & 59                                                               & 1                   &                                                                   &                                                                          \\
bankruptcy       & 49, 54                                                           & 1                   &                                                                   &                                                                          \\
       & 50                                                               & 3                   & 0.92                                                              & 4                                                                        \\
       & 51                                                               & 9                   & 0.92                                                              & 4.1                                                                      \\
       & 52                                                               & 10                  & 0.94                                                              & 3.1                                                                      \\
       & 53                                                               & 8                   & 0.96                                                              & 2                                                                        \\
Ovarian          & 34                                                               & 2                   & 0.84                                                              & 6                                                                        \\
          & 35                                                               & 6                   & 0.89                                                              & 4.3                                                                      \\
          & 36                                                               & 5                   & 0.92                                                              & 3                                                                        \\
          & 37                                                               & 3                   & 0.95                                                              & 2                                                                        \\
          & 38                                                               & 1                   &                                                                   &                                                                          \\
p53              & 71, 72                                                           & 1                   &                                                                   &                                                                          \\
              & 73                                                               & 7                   & 0.93                                                              & 5.3                                                                      \\
              & 74                                                               & 18                  & 0.94                                                              & 5                                                                        \\
              & 75                                                               & 18                  & 0.95                                                              & 3.6                                                                      \\
              & 76                                                               & 14                  & 0.97                                                              & 2                                                                        \\
              & 77                                                               & 1                   &                                                                   &                                                                          \\
Parkinson        & 12, 13                                                           & 1                   &                                                                   &                                                                          \\
prostate         & 21, 22, 24                                                       & 1                   &                                                                   &                                                                          \\
         & 23                                                               & 3                   & 0.92                                                              & 2                                                                        \\
secom            & 41, 48                                                           & 1                   &                                                                   &                                                                          \\
            & 43                                                               & 2                   & 0.83                                                              & 8                                                                        \\
            & 44                                                               & 14                  & 0.88                                                              & 5.9                                                                      \\
            & 45                                                               & 29                  & 0.9                                                               & 4.8                                                                      \\
            & 46                                                               & 21                  & 0.93                                                              & 3.4                                                                      \\
            & 47                                                               & 12                  & 0.96                                                              & 2                                                                        \\
vV               & 10, 12                                                           & 1                   &                                                                   &                                                                          \\\hline
BC\_Continuous   & 12, 15                                                           & 1                   &                                                                   &                                                                          \\
   & 13                                                               & 4                   & 0.86                                                              & 2                                                                        \\
   & 14                                                               & 6                   & 0.87                                                              & 2                                                                        \\
MP               & 104, 109                                                         & 1                   &                                                                   &                                                                          \\
               & 105                                                              & 8                   & 0.94                                                              & 7                                                                        \\
               & 106                                                              & 18                  & 0.95                                                              & 5.1                                                                      \\
               & 107                                                              & 30                  & 0.97                                                              & 3.6                                                                      \\
               & 108                                                              & 17                  & 0.98                                                              & 2                                                                       \\\hline
\multicolumn{5}{c}{SES} \\ \hline
arcene           & 5                                                                   & 100                    & 0.37                                                              & 4.8                                                                      \\
bankruptcy       & 17                                                                  & 6                      & 0.85                                                              & 2.8                                                                      \\
Ovarian          & 6                                                                   & 6                      & 0.63                                                              & 2.8                                                                      \\
Parkinson        & 1                                                                   & 2                      & 0                                                                 & 2                                                                        \\
prostate         & 3                                                                   & 3                      & 0.5                                                               & 2                                                                        \\
secom            & 15                                                                  & 100                    & 0.6                                                               & 7.7                                                                      \\
vV               & 4                                                                   & 4                      & 0.51                                                              & 2.7                                                                      \\\hline
BC\_Continuous   & 12                                                                  & 6                      & 0.79                                                              & 2.8                                                                      \\
MP               & 17                                                                  & 24                     & 0.79                                                              & 4         
\end{tabular}
\end{table}

\begin{figure}[t]
\begin{center}
\includegraphics[width=0.5\textwidth]{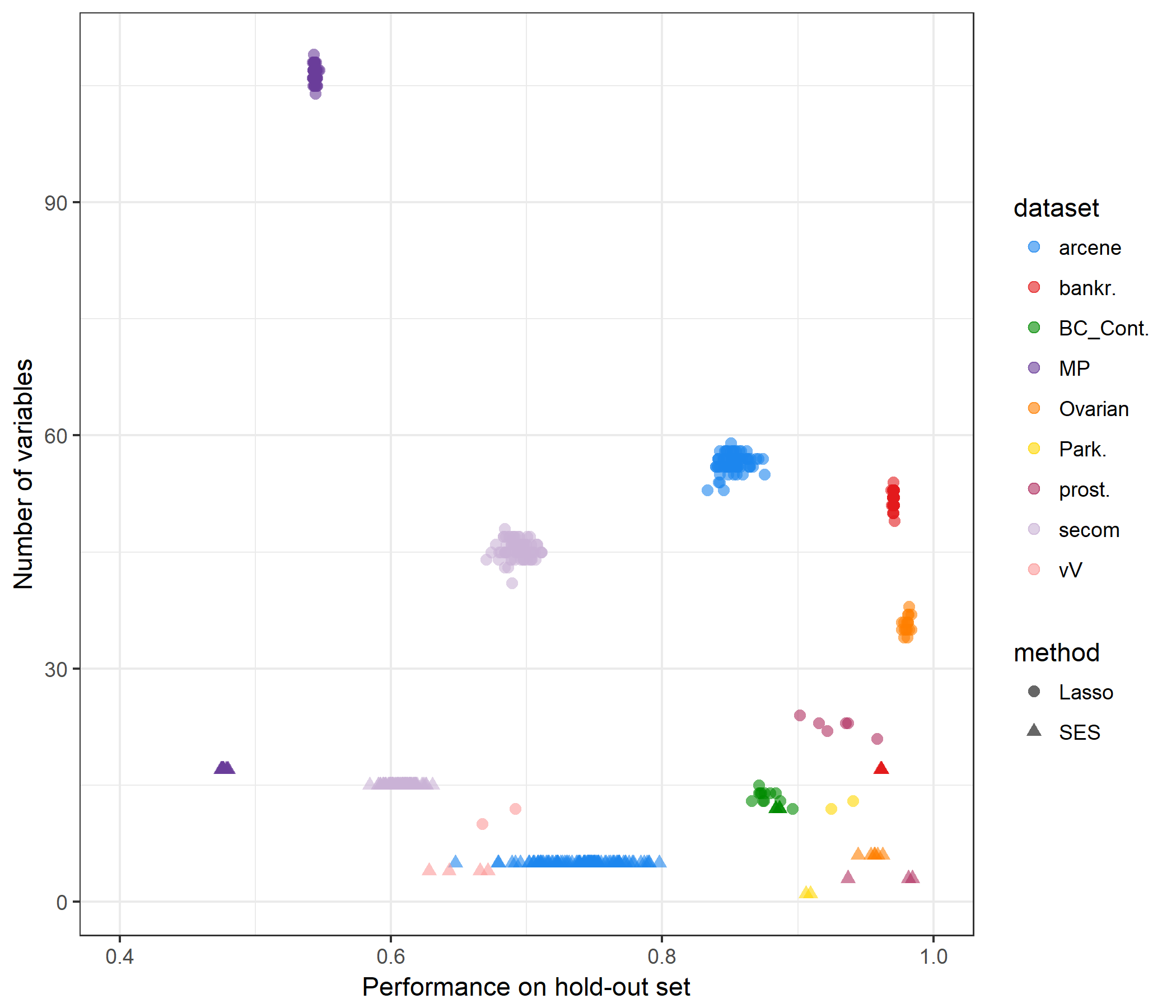}
\caption{Lasso and SES results. Each point represents a single solution, whose hold-out performance is reported on the x-axis (AUC for classification datasets, R$^2$ for regression ones), while the number of elements in each signature is reported on the y-axis. Triangles represent SES' solutions, circles Lasso', and different datasets are reported in different colors. For sake of clarity only the first split for each dataset is reported; results on second splits are similar (not shown).}
\label{figure:classificationResults}
\end{center}
\end{figure}
\vspace{-2mm}
\section*{Acknowledgments}
\vspace{-1mm}
We would like to thank Giorgos Borboudakis for his constructive discussions. The research leading to these results has received funding from the European Research Council under the European Union's Seventh Framework Programme (FP/2007-2013) / ERC Grant Agreement n. 617393. PC was funded by Gnosis Data Analysis PC.

\bibliographystyle{unsrt}
\bibliography{biblio}

\appendix
\section{Proof of Theorem \ref{RMSE:DEV:thm}}

\begin{proof}
\noindent {\bf  (a).} First, apply the Minkowski inequality as follows
\begin{equation*} \small 
\begin{aligned} 
||y-X_{\mathcal E} \bar{\beta}_{\mathcal E}||_2
&= ||y-X_{\mathcal E} \hat{\beta}_{\mathcal E} + X_{\mathcal E} (\hat{\beta}_{\mathcal E}-\bar{\beta}_{\mathcal E})||_2 \\
&\leq ||y-X_{\mathcal E} \hat{\beta}_{\mathcal E}||_2 + ||X_{\mathcal E} (\hat{\beta}_{\mathcal E}-\bar{\beta}_{\mathcal E}) ||_2
\end{aligned}
\end{equation*}

Next, define the matrices $\bar{V}_{\mathcal E}^* := [v_{i^*+1}|...|v_{|\mathcal E|}]$,
$$
\Sigma_{\mathcal E}^* = \begin{bmatrix} diag(\sigma_1,...,\sigma_{i^*}) & 0 \\ 0 & 0
\end{bmatrix},
\ \ \text{and,} \ \ 
\bar{\Sigma}_{\mathcal E}^* = \begin{bmatrix} 0 & 0 \\ 0 & diag(\sigma_{i^*+1},...,\sigma_{|\mathcal E|})
\end{bmatrix}
$$
It holds that $V_{\mathcal E}=[V_{\mathcal E}^*|\bar{V}_{\mathcal E}^*]$ as well as $\Sigma_{\mathcal E} = \Sigma_{\mathcal E}^* + \bar{\Sigma}_{\mathcal E}^*$. Using these matrices we write

\begin{equation*}\small
\begin{aligned}
X_{\mathcal E} \bar{\beta}_{\mathcal E} &= U_{\mathcal E} \Sigma_{\mathcal E} V_{\mathcal E}^T \bar{\beta}_{\mathcal E} = U_{\mathcal E} \Sigma_{\mathcal E}^* (V_{\mathcal E}^*)^T \bar{\beta}_{\mathcal E} + U_{\mathcal E} \bar{\Sigma}_{\mathcal E}^* (\bar{V}_{\mathcal E}^*)^T \bar{\beta}_{\mathcal E} \\
&= U_{\mathcal E} \Sigma_{\mathcal E}^* (V_{\mathcal E}^*)^T \hat{\beta}_{\mathcal E} + U_{\mathcal E} \bar{\Sigma}_{\mathcal E}^* (\bar{V}_{\mathcal E}^*)^T \bar{\beta}_{\mathcal E} \\
&= U_{\mathcal E} \Sigma_{\mathcal E} V_{\mathcal E}^T \hat{\beta}_{\mathcal E} + U_{\mathcal E} \bar{\Sigma}_{\mathcal E}^* (\bar{V}_{\mathcal E}^*)^T (\bar{\beta}_{\mathcal E}-\hat{\beta}_{\mathcal E}) \\
&= X_{\mathcal E} \hat{\beta}_{\mathcal E} + U_{\mathcal E} \bar{\Sigma}_{\mathcal E}^* (\bar{V}_{\mathcal E}^*)^T (\bar{\beta}_{\mathcal E}-\hat{\beta}_{\mathcal E})
\end{aligned}
\end{equation*}

Hence, we bound the error term in the inequality above by
\begin{equation*}\small
\begin{aligned}
||X_{\mathcal E} (\hat{\beta}_{\mathcal E}-\bar{\beta}_{\mathcal E}) ||_2
&= ||U_{\mathcal E} \bar{\Sigma}_{\mathcal E}^* (\bar{V}_{\mathcal E}^*)^T (\bar{\beta}_{\mathcal E}-\hat{\beta}_{\mathcal E})||_2 \\
&\leq ||U_{\mathcal E} \bar{\Sigma}_{\mathcal E}^* (\bar{V}_{\mathcal E}^*)^T||_2 ||\bar{\beta}_{\mathcal E}-\hat{\beta}_{\mathcal E}||_2
\end{aligned}
\end{equation*}
where we used the vector-induced matrix norm $||A||_2 := \sup_{x\neq0} \frac{||Ax||_2}{||x||_2}$ which is also known as the spectral norm. It holds that $||A||_2 = \sigma_{max}(A)$ thus the bound is rewritten as
\begin{equation*}
||X_{\mathcal E} (\hat{\beta}_{\mathcal E}-\bar{\beta}_{\mathcal E}) ||_2 \le \sigma_{i^*+1} ||\bar{\beta}_{\mathcal E}-\hat{\beta}_{\mathcal E}||_2
\end{equation*}

Due to the singular value ordering, it holds that $\sigma_{i^*+1} = ||\bar{\sigma}^*||_\infty=$. The proof is completed by observing that the box constraint leads to the estimate
$$
||\bar{\beta}_{\mathcal E}-\hat{\beta}_{\mathcal E}||_2 \le  2l\sqrt{|\mathcal E|} \ .
$$


\noindent {\bf  (b).}
We observe that for all $i\in\mathcal E$
$$ \footnotesize
x_{\mathcal E,i} = \sum_{j=1}^{|\mathcal E|} c_{ij} v_j \ ,
$$
where $c_{ij} = x_{\mathcal E,i}^Tv_j$ is the projection coefficient of the $i$-th sample to the $j$-th principal vector. Thus, it holds that
\begin{equation*} \small
\begin{aligned}
x_i^T(\bar{\beta}-\hat{\beta}) &= x_{\mathcal E,i}^T(\bar{\beta}_{\mathcal E}-\hat{\beta}_{\mathcal E})
= \sum_{j=1}^{|\mathcal E|} c_{ij} v_j^T(\bar{\beta}_{\mathcal E}-\hat{\beta}_{\mathcal E}) \\
&= \sum_{j=1}^{i^*} c_{ij} v_j^T(\bar{\beta}_{\mathcal E}-\hat{\beta}_{\mathcal E}) + \sum_{j=i^*+1}^{|\mathcal E|} c_{ij} v_j^T(\bar{\beta}_{\mathcal E}-\hat{\beta}_{\mathcal E}) \\
&= \sum_{j=i^*+1}^{|\mathcal E|} c_{ij} v_j^T(\bar{\beta}_{\mathcal E}-\hat{\beta}_{\mathcal E}) =: e_i
\end{aligned}
\end{equation*}
because $\bar{\beta}\in K^*$ implies that $v_j^T\bar{\beta}_{\mathcal E}=v_j^T\hat{\beta}_{\mathcal E}$ for $j=1,...,i^*$.
Proceeding, the convexity of the function $\log(1+\exp(\cdot))$ implies that
$$
\log(1+\exp(x_i^T\bar{\beta})) \leq \log(1+\exp(x_i^T\hat{\beta})) + e_i\frac{\exp(x_i^T\bar{\beta})}{1+\exp(x_i^T\bar{\beta})} \ .
$$
Hence,
\vspace{-2mm}
\begin{equation*}\small
\begin{aligned}
DEV(\bar{\beta}) &\leq DEV(\hat{\beta}) + \frac{1}{n} \sum_{i=1}^n e_i \left( \frac{\exp(x_i^T\bar{\beta})}{1+\exp(x_i^T\bar{\beta})} - y_i\right) \\
&\leq DEV(\hat{\beta}) + \frac{1}{n} | e^T z| \ ,
\end{aligned}
\end{equation*}
where $e$ and $z$ are $n$-dimensional vectors with elements $e_i$ and $z_i = \frac{\exp(x_i^T\bar{\beta})}{1+\exp(x_i^T\bar{\beta})} - y_i$. Using the Cauchy-Schwarz inequality, we get
$$
DEV(\bar{\beta}) \leq DEV(\hat{\beta}) + \frac{1}{n} ||e||_2 ||z||_2
\leq DEV(\hat{\beta}) + \frac{1}{\sqrt{n}} ||e||_2
$$
where the last inequality is the consequence of the fact that $|z_i|\leq 1$ for all $i=1,...,n$, hence, $||z||_2 \leq \sqrt{n}$.

Finally, we estimate
\begin{equation*}\small
\begin{aligned}
||e||_2^2 &= \sum_{i=1}^n \left( \sum_{j=i^*+1}^{|\mathcal E|} c_{ij} v_j^T(\bar{\beta}_{\mathcal E}-\hat{\beta}_{\mathcal E})\right)^2
= \sum_{i=1}^n \sum_{j=i^*+1}^{|\mathcal E|} c_{ij}^2 ||\bar{\beta}_{\mathcal E}-\hat{\beta}_{\mathcal E}||_2^2 \\
&= ||\bar{\beta}_{\mathcal E}-\hat{\beta}_{\mathcal E}||_2^2 \sum_{j=i^*+1}^{|\mathcal E|} \sum_{i=1}^n c_{ij}^2 
= ||\bar{\beta}_{\mathcal E}-\hat{\beta}_{\mathcal E}||_2^2 \sum_{j=i^*+1}^{|\mathcal E|} \sigma_j^2 \\
&= ||\bar{\beta}_{\mathcal E}-\hat{\beta}_{\mathcal E}||_2^2 ||\bar{\sigma}^*||_2^2 \ ,
\end{aligned}
\end{equation*}
where we use the fact that $\sigma_j=||X_{\mathcal E} v_j||_2$. Combining the above with the bound for the difference $(\bar{\beta}_{\mathcal E}-\hat{\beta}_{\mathcal E})$, we obtain the desired bound.
\end{proof}

\end{document}